\documentclass[12pt]{article}
\usepackage{amssymb}
\usepackage{amsthm}
\usepackage{amsmath}
\usepackage{graphicx}
\usepackage{enumerate}
\usepackage{pict2e}
\usepackage{framed}

\DeclareMathOperator{\Con}{Con}

\DeclareMathOperator{\BCon}{\mathbf{Con}}

\newtheorem{theorem}{Theorem}

\newtheorem{lemma}[theorem]{Lemma}
\newtheorem{proposition}[theorem]{Proposition}

\newtheorem{remark}[theorem]{Remark}
\newtheorem{example}[theorem]{Example}
\newtheorem{corollary}[theorem]{Corollary}
\title{The variety of complemented lattices where the Sasaki operations form an adjoint pair}
\author{V\'aclav~Cenker, Ivan~Chajda and Helmut~L\"anger}
\date{}

\begin{document}

\footnotetext{Support of the research of the first two authors by the Czech Science Foundation (GA\v CR), project 24-14386L, entitled ``Representation of algebraic semantics for substructural logics'', and by IGA, project P\v rF~2024~011, is gratefully acknowledged.}
	
\maketitle
	
\begin{abstract}
The Sasaki projection was introduced as a mapping from the lattice of closed subspaces of a Hilbert space onto one of its segments. To use this projection and its dual so-called Sasaki operations were introduced by the second two authors in \cite{CL17} and \cite{CLa}. In \cite{CLa} there are described several classes of lattices, $\lambda$-lattices and semirings where the Sasaki operations form an adjoint pair. In the present paper we prove that the class of complemented lattices with this property forms a variety and we explicitly state its defining identities. Moreover, we prove that this variety $\mathcal V$ is congruence permutable and regular. Hence every ideal $I$ of some member $\mathbf L$ of $\mathcal V$ is a kernel of some congruence on $\mathbf L$. Finally, we determine a finite basis of so-called ideal terms and describe the congruence $\Theta_I$ determined by the ideal $I$.
\end{abstract}
	
	

It was proved by G.~Birkhoff and J.~von~Neumann \cite{BV} in 1936, and independently by K.~Husimi \cite H in 1937 that the lattice $\mathcal L(\mathbf H)$ of closed subspaces of a Hilbert space $\mathbf H$ is orthomodular. For a given element $a\in\mathcal L(\mathbf H)$, U.~Sasaki \cite S introduced the projection $p_a$ of $\mathcal L(\mathbf H)$ onto its interval $[0,a]$ as follows:
\[
p_a(x)=(x\vee a')\wedge a.
\]
The name {\em Sasaki projection} was given later by M.~Nakamura \cite N. Also, we can formulate its dual, i.e.\ the projection
\[
\overline{p_a}(x)=(x\wedge a)\vee a'
\]
of $\mathcal L(\mathbf H)$ onto its interval $[a',1]$.

However, we can extend these definitions by introducing the corresponding concepts in every bounded lattice $(L,\vee,\wedge, {}',0,1)$ equipped with a unary operation $'$. Using this framework, we define two binary operations on $L$.

Let $\mathbf L=(L,\vee,\wedge,{}',0,1)$ be a bounded lattice with unary operation $'$. By the {\em Sasaki operations} on $\mathbf L$ we mean the following term operations:
\[
x\odot y=(x\vee y')\wedge y\quad\text{and}\quad x\to y=x'\vee(x\wedge y).
\]
We say that they form an {\em adjoint pair} on $\mathbf L$ if the following holds:
\begin{enumerate}
\item[(A)] $x\odot y\le z$ if and only if $x\le y\to z$.
\end{enumerate}

\begin{lemma}
Let $\mathbf L=(L,\vee,\wedge,{}',0,1)$ be a bounded lattice with unary operation $'$. If the Sasaki operations $\odot$, $\to$ form an adjoint pair in $\mathbf L$, then the unary operation $'$ is a complementation.
\end{lemma}

\begin{proof}
Because $1\odot a=a\le1$ for each $a\in L$, by (A), we derive $1\le a\to1=a'\vee(a\wedge1)=a'\vee a$. Since $0\le a\to0$, we also have $a'\wedge a=(0\vee a')\wedge a=0\odot a\le0$, i.e.\ $a'\wedge a=0$. Thus, $'$ is a complementation.
\end{proof}

Hence, investigating whether the Sasaki operations in lattices with a unary operation form an adjoint pair makes sense only for lattices with complementation. At first we determine the class of such lattices.

\begin{theorem}\label{th1}
Let $\mathbf L=(L,\vee,\wedge,{}',0,1)$ be a bounded lattice with complementation. Then the following are equivalent:
\begin{enumerate}
\item[{\rm(a)}] $x\odot y\le z$ implies $x\le y\to z$ for all $x,y,z\in L$,
\item[{\rm(b)}] $\mathbf L$ satisfies the identity $x\vee y'\approx y'\vee\big((x\vee y')\wedge y\big)$,
\item[{\rm(c)}] $x'\le y$ implies $y=x'\vee(y\wedge x)$.
\end{enumerate}
Furthermore, the following are equivalent:
\begin{enumerate}
\item[{\rm(d)}] $x\le y\to z$ implies $x\odot y\le z$ for all $x,y,z\in L$,
\item[{\rm(e)}] $\mathbf L$ satisfies the identity $x\wedge y\approx x\wedge\big((x\wedge y)\vee x'\big)$,
\item[{\rm(f)}] $x\le y$ implies $x=(y'\vee x)\wedge y$.
\end{enumerate}
\end{theorem}

\begin{proof}
$\text{}$ \\
(a) $\Rightarrow$ (c): \\
Assume $x'\le y$ for $x,y\in L$. Then, $y\odot x=(y\vee x')\wedge x=y\wedge x$. According to (a) we obtain
\[
y\le x\to(y\wedge x)=x'\vee(y\wedge x)\le y\vee(y\wedge x)=y
\]
and hence $y=x'\vee(y\wedge x)$. \\
(c) $\Rightarrow$ (b): \\
Because $y'\le x\vee y'$, we apply (c) to get $x\vee y'=y'\vee\big((x\vee y')\wedge y)$. \\
(b) $\Rightarrow$ (a): \\
Assume $x\odot y\le z$. Using (b) we derive 
\begin{align*}
x & \le x\vee y'=y'\vee\big((x\vee y')\wedge y\big)=y'\vee\Big(\big((x\vee y')\wedge y\big)\wedge y\Big)=y'\vee\big((x\odot y)\wedge y\big)\le \\
  & \le y'\vee(z\wedge y)=y\to z.
\end{align*}
Similarly, we prove the remaining. \\
(d) $\Rightarrow$ (f): \\
Let $x\le y$. Then $y\to x=y'\vee(y\wedge x)=y'\vee x$. By (d) we have $(y'\vee x)\odot y\le x$ whence
\[
(y'\vee x)\wedge y=\big((y'\vee x)\big)\vee y')\wedge y=(y'\vee x)\odot y\le x.
\]
Conversely, from the assumption and fact that $x\le y'\vee x$ we conclude $x\le(y'\vee x)\wedge y $. Consequently, $x=(y'\vee x)\wedge y$. \\
(f) $\Rightarrow$ (e): \\
Because $x\wedge y\le x$, we apply (f) to get $x\wedge y=x\wedge\big((x\wedge y)\vee x'\big)$. \\
(e) $\Rightarrow$ (d): \\
Assume $x\le y\to z$. Then, using (e) we obtain
\begin{align*}
x\odot y & =(x\vee y')\wedge y\le\big((y \to z)\vee y'\big)\wedge y=\Big(\big(y'\vee(y\wedge z)\big)\vee y'\Big)\wedge y= \\
         & =\big(y'\vee(y\wedge z)\big)\wedge y=y\wedge z\le z.
\end{align*}
\end{proof}

\begin{corollary}
The class of bounded lattices with a unary operation $\,'$ where the Sasaki operations form an adjoint pair is a variety determined by the identities of complemented lattices and the identities {\rm(b)} and {\rm(c)} from Theorem~\ref{th1}. Let us denote this variety by $\mathcal V$.
\end{corollary}

\begin{remark}
It is evident that the aforementioned variety $\mathcal V$ includes the variety of orthomodular lattices, but does not coincide with it. For instance, the modular lattice $\mathbf M_3$ depicted in Figure~1

\vspace*{-8mm}

\begin{center}
\setlength{\unitlength}{7mm}
\begin{picture}(6,6)
\put(3,1){\circle*{.3}}
\put(1,3){\circle*{.3}}
\put(3,3){\circle*{.3}}
\put(5,3){\circle*{.3}}
\put(3,5){\circle*{.3}}
\put(3,1){\line(-1,1)2}
\put(3,1){\line(0,1)4}
\put(3,1){\line(1,1)2}
\put(3,5){\line(-1,-1)2}
\put(3,5){\line(1,-1)2}
\put(2.85,.3){$0$}
\put(.35,2.85){$a$}
\put(3.4,2.85){$b$}
\put(5.4,2.85){$c$}
\put(2.85,5.4){$1$}
\put(2.2,-.75){{\rm Fig.~1}}
\put(.4,-1.75){{\rm Modular lattice $\mathbf M_3$}}
\end{picture}
\end{center}

\vspace*{10mm}

with complementation defined by the table
\[
\begin{array}{l|ccccc}
x  & 0 & a & b & c & 1 \\
\hline
x' & 1 & b & c & a & 0
\end{array}
\]
belongs to $\mathcal V$ but it is not an orthomodular lattice since the complementation is not an involution since $(a')'=b'=c\ne a$. However, the complementation of $\mathbf M_3$ is antitone. Moreover, $\mathcal V$ contains also non-modular lattices since not every orthomodular lattice is modular. Further, $\mathcal V$ includes the variety of complemented modular lattices since within these lattices we have
\begin{align*}
   y'\vee\big(y\wedge(x\vee y')\big) & \approx(y'\vee y)\wedge(x\vee y')\approx x\vee y', \\
\big((x\wedge y)\vee x'\big)\wedge x & \approx(x\wedge y)\vee(x'\wedge x)\approx x\wedge y.
\end{align*}
Finally, $\mathcal V$ is contained in the variety of dually weakly orthomodular lattices, and every member of $\mathcal V$ whose complementation is both surjective and an involution is a weakly orthomodular lattice. For more details on {\rm(}dually{\rm)} weakly orthomodular lattices cf.\ {\rm\cite{CL18}}.
\end{remark}

\begin{example}\label{ex1}
The modular lattice visualized in Figure~2

\vspace*{-2mm}

\begin{center}
\setlength{\unitlength}{7mm}
\begin{picture}(12,8)
\put(3,1){\circle*{.3}}
\put(1,3){\circle*{.3}}
\put(3,3){\circle*{.3}}
\put(5,3){\circle*{.3}}
\put(9,3){\circle*{.3}}
\put(3,5){\circle*{.3}}
\put(7,5){\circle*{.3}}
\put(9,5){\circle*{.3}}
\put(11,5){\circle*{.3}}
\put(9,7){\circle*{.3}}
\put(3,1){\line(-1,1)2}
\put(3,1){\line(0,1)4}
\put(3,1){\line(1,1)2}
\put(3,1){\line(3,1)6}
\put(1,3){\line(1,1)2}
\put(1,3){\line(3,1)6}
\put(3,3){\line(3,1)6}
\put(5,3){\line(-1,1)2}
\put(5,3){\line(3,1)6}
\put(9,3){\line(-1,1)2}
\put(9,3){\line(0,1)4}
\put(9,3){\line(1,1)2}
\put(3,5){\line(3,1)6}
\put(7,5){\line(1,1)2}
\put(11,5){\line(-1,1)2}
\put(2.85,.3){$0$}
\put(.35,2.85){$a$}
\put(2.35,2.85){$b$}
\put(4.35,2.85){$c$}
\put(8.85,2.3){$d$}
\put(2.85,5.4){$e$}
\put(7.4,4.85){$f$}
\put(9.4,4.85){$g$}
\put(11.4,4.85){$h$}
\put(8.85,7.4){$1$}
\put(5.3,-.75){{\rm Fig.~2}}
\put(4.4,-1.75){{\rm Member of $\mathcal V$}}
\end{picture}
\end{center}

\vspace*{10mm}

with complementation defined by the table
\[
\begin{array}{l|cccccccccc}
x  & 0 & a & b & c & d & e & f & g & h & 1 \\
\hline
x' & 1 & g & h & f & e & d & c & a & b & 0
\end{array}
\]
is a member of $\mathcal V$. Observe that the complementation of this lattice is an involution but not antitone since $a\le f$, but $f'=c\not\le g=a'$.
\end{example}

In what follows we describe some important congruence properties of the variety $\mathcal V$.

Recall the following facts concerning congruence properties.

Let $\mathbf A=(A,F)$ be an algebra and $\BCon\mathbf A=(\Con\mathbf A,\subseteq)$ denote its congruence lattice. Then $\mathbf A$ is called
\begin{enumerate}[(i)]
\item {\em congruence permutable} if $\Theta\circ\Phi=\Phi\circ\Theta$ for all $\Theta,\Phi\in\Con\mathbf A$,
\item {\em congruence distributive} if $\BCon\mathbf A$ is distributive,
\item {\em arithmetical} if it is both congruence permutable and congruence distributive,
\item {\em regular} if every congruence on $\mathbf A$ is determined by each of its classes, i.e.\ if $a\in A$, $\Theta,\Phi\in\Con\mathbf A$ and $[a]\Theta=[a]\Phi$ together imply $\Theta=\Phi$,	
\end{enumerate}
see e.g.\ \cite{CEL} and \cite C. A variety is said to have one of the above congruence properties if every of its members has this property. The following results are well-known:
\begin{enumerate}[(i)]
\item A variety is congruence permutable if and only if there exists a so-called Mal'cev term, i.e.\ a ternary term $p$ satisfying the identities $p(x,x,z)\approx z$ and $p(x,z,z)\approx x$,
\item if in a variety there exists a so-called majority term, i.e.\ a ternary term $m$ satisfying the identities $m(x,x,z)\approx m(x,y,x)\approx x$ and $m(x,z,z)\approx z$ then this variety is congruence distributive,
\item a variety is regular if and only if there exists some positive integer $n$ and ternary terms $t_1,\ldots,t_n$ such that $t_1(x,y,z)=\cdots=t_n(x,y,z)=z$ if and only if $x=y$.
\end{enumerate}

\begin{theorem}\label{th2}
The variety $\mathcal V$ is arithmetical and regular.
\end{theorem}

\begin{proof}
Since $\mathcal V$ is a variety of lattices, it is congruence distributive (namely, it contains the majority term $m(x,y,z):=(x\vee y)\wedge(y\vee z)\wedge(z\vee x)$). We prove congruence permutability of $\mathcal V$. Consider the term
\[
p(x,y,z):=\Big(x\wedge\big((z\wedge y)\vee y'\big)\Big)\vee\Big(z\wedge\big((x\wedge y)\vee y'\big)\Big).
\]
Using (f) of Theorem~\ref{th1} we compute
\begin{align*}
p(x,x,z) & \approx\Big(x\wedge\big((z\wedge x)\vee x'\big)\Big)\vee\Big(z\wedge\big((x\wedge x)\vee x'\big)\Big)\approx(z\wedge x)\vee z\approx z, \\
p(x,z,z) & \approx\Big(x\wedge\big((z\wedge z)\vee z'\big)\Big)\vee\Big(z\wedge\big((x\wedge z)\vee z'\big)\Big)\approx x\vee(x\wedge z)\approx x.
\end{align*}
Hence $p$ is a Mal'cev term. Thus $\mathcal V$ is congruence permutable and therefore arithmetical. In order to prove regularity of $\mathcal V$ take $n:=2$ and consider the ternary terms
\begin{align*}
t_1(x,y,z) & :=u\wedge z, \\
t_2(x,y,z) & :=u'\vee z
\end{align*}
where $u:=(x\wedge y)\vee(x\vee y)'$. Evidently, $t_1(x,y,z)=t_2(x,y,z)=z$ provided $x=y$. Conversely assume $t_1(x,y,z)=t_2(x,y,z)=z$. Then 
\[
u=u\vee(u\wedge z)=u\vee t_1(x,y,z)=u\vee t_2(x,y,z)=u\vee(u'\vee z)=1.
\]
Since $x\wedge y\le x\vee y$, we use (f) of Theorem~\ref{th1} to get
\[
x\wedge y=(x\vee y)\wedge\big((x\wedge y)\vee(x\vee y)'\big)=(x\vee y)\wedge u=(x\vee y)\wedge1=x\vee y
\]
which together with $x\wedge y\le x,y\le x\vee y$ yields $x=y$.
\end{proof}

Using advantage of congruence properties of the variety $\mathcal V$, we go on with a finer analysis of congruences and their classes.

Recall the following definitions.

Let $\mathbf L=(L,\vee,\wedge,{}',0,1)$ be some member of $\mathcal V$ and $I\subseteq L$. An {\em ideal term in $y_1,\ldots,y_m$} is a term $t(x_1,\ldots,x_n,y_1,\ldots,y_m)$ satisfying the identity
\[
t(x_1,\ldots,x_n,1,\ldots,1)\approx1
\]
The subset $I$ of $L$ is called {\em closed under the ideal term $t(x_1,\ldots,x_n,y_1,\ldots,y_m)$ in $y_1,\ldots,y_m$} if
\[
t(a_1,\ldots,a_n,b_1,\ldots,b_m)\in I\text{ for all }a_1,\ldots,a_n\in L\text{ and all }b_1,\ldots,b_m\in I.
\]
Moreover, $I$ is called an {\em ideal} of $\mathbf L$ if it is closed under all ideal terms. It is well-known (cf.\ Lemma~10.1.3 and Theorem~10.1.10 of \cite{CEL}) that in a congruence permutable variety ideals and congruence kernels coincide. A set $T$ of ideal terms is called a {\em basis of ideal terms} if a subset of the base set of a lattice $\mathbf L$ belonging to $\mathcal V$ is an ideal of $\mathbf L$ if and only if it is closed under the ideal terms of $T$.

As stated in Theorem~\ref{th2}, the variety $\mathcal V$ is regular and hence also weakly regular, i.e.\ for each lattice $\mathbf L$ belonging in $\mathcal V$, every congruence $\Theta$ on $\mathbf L$ is determined by its class $[1]\Theta$, i.e.\ by its ideal. Our next task is to determine these ideals by using so-called ideal terms. In order to be able to describe a finite basis of ideal terms of the variety $\mathcal V$ we apply Theorem~10.3.4 from \cite{CEL}. But first we need Theorem~6.4.11 from \cite{CEL}.

\begin{proposition}\label{prop1}
{\rm(\cite{CEL}}, Theorem~6.4.11{\rm)} For a variety $\mathcal V$ with $1$ the following conditions are equivalent:
\begin{enumerate}[{\rm(i)}]
\item $\mathcal V$ is permutable and weakly regular.
\item There exists a positive integer $n$, binary terms $t_1(x,y),\ldots,t_n(x,y)$ and an $(n+2)$-ary term $t(x_1,\ldots,x_{n+2})$ satisfying the following identities:
\begin{align*}
& t_1(x,x)\approx\cdots\approx t_n(x,x)\approx1, \\
& t\big(x,y,t_1(x,y),\ldots,t_n(x,y)\big)\approx x, \\
& t(x,y,1,\ldots,1)\approx y.
\end{align*}
\end{enumerate}	
\end{proposition}

In the next result the terms from Proposition~\ref{prop1} (ii) are applied.

\begin{proposition}\label{prop2}
{\rm(\cite{CEL}}, Theorem~10.3.4{\rm)} For a permutable and weakly regular variety $\mathcal V$ the set consisting of $1$ and of the following ideal terms in the $y$'s is a basis of ideal terms:
\begin{enumerate}[{\rm(i)}]
\item $t_i\Big(f\big(t(x_1,x_1',y_1^1,\ldots,y_n^1),\ldots,t(x_m,x_m',y_1^m,\ldots,y_n^m)\big),f(x_1',\ldots,x_m')\Big)$ for each positive integer m, each $m$-ary fundamental operation of $\mathcal V$ and all $i=1,\ldots,n$,
\item $t(x,1,y_1,\ldots,y_n)$,
\item $t_1(y,1),\ldots,t_n(y,1)$.
\end{enumerate}
{\rm(}Here $t_1,\ldots,t_n,t$ denote the terms occurring in Proposition~\ref{prop1}.{\rm)}
\end{proposition}

In the sequel we will need also the following two terms:
\begin{align*}
	t_1(x,y) & :=(x\wedge y)\vee(x\vee y)', \\
	t(x,y,z) & :=\Bigg(\bigg(\Big(\big((x\vee y)\wedge z\big)'\vee x\Big)\wedge(x\vee y)\bigg)'\vee x\Bigg)\wedge\Big(\big((x\vee y)\wedge z\big)'\vee y\Big)\wedge(x\vee y).
\end{align*}

\begin{lemma}\label{lem1}
The variety $\mathcal V$ satisfies the following identities:
\begin{enumerate}[{\rm(i)}]
\item $x\vee(x\wedge y)'\approx1$,
\item $y\vee(x\wedge y)'\approx1$,
\item $x\approx(x\vee y)\wedge\big(x\vee(x\vee y)'\big)$,
\item $y\approx(x\vee y)\wedge\big(y\vee(x\vee y)'\big)$,
\item $(x\vee y)\wedge t_1(x,y)\approx x\wedge y$,
\item $t_1(x,x)\approx1$,
\item $t\big(x,y,t_1(x,y)\big)\approx x$,
\item $t(x,y,1)\approx  y$.
\end{enumerate}
\end{lemma}

\begin{proof}
\
\begin{enumerate}
\item[(i)] We have $x\vee(x\wedge y)'\approx x\vee(x\wedge y)\vee(x\wedge y)'\approx1$.
\item[(ii)] This follows from (i).
\item[(iii)] and (iv) follow from Theorem~\ref{th1} (f) since $x\le x\vee y$ and $y\le x\vee y$.
\item[(v)] According to Theorem~\ref{th1} (f) we have
\[
(x\vee y)\wedge t_1(x,y)\approx(x\vee y)\wedge\big((x\wedge y)\vee(x\vee y)'\big)\approx x\wedge y.
\]
\item[(vi)] We have $t_1(x,x)\approx(x\wedge x)\vee(x\vee x)'\approx x\vee x'\approx1$.
\item[(vii)] According to (i), (ii), (iii) and (v) we have
\begin{align*}
t\big(x,y,t_1(x,y)\big) & \approx\Bigg(\bigg(\Big(\big((x\vee y)\wedge t_1(x,y)\big)'\vee x\Big)\wedge(x\vee y)\bigg)'\vee x\Bigg)\wedge \\
& \hspace*{6mm}\wedge\Big(\big((x\vee y)\wedge t_1(x,y)\big)'\vee y\Big)\wedge(x\vee y)\approx \\
& \approx\bigg(\Big(\big((x\wedge y)'\vee x\big)\wedge(x\vee y)\Big)'\vee x\bigg)\wedge\big((x\wedge y)'\vee y\big)\wedge(x\vee y)\approx \\
& \approx\Big(\big(1\wedge(x\vee y)\big)'\vee x\Big)\wedge1\wedge(x\vee y)\approx\big((x\vee y)'\vee x\big)\wedge(x\vee y)\approx x.
\end{align*}
\item[(viii)] According to (iii) and (iv) we have
\begin{align*}
t(x,y,1) & \approx\Bigg(\bigg(\Big(\big((x\vee y)\wedge1\big)'\vee x\Big)\wedge(x\vee y)\bigg)'\vee x\Bigg)\wedge\Big(\big((x\vee y)\wedge1\big)'\vee y\Big)\wedge \\
& \hspace*{6mm}\wedge(x\vee y)\approx \\
& \approx\bigg(\Big(\big((x\vee y)'\vee x\big)\wedge(x\vee y)\Big)'\vee x\bigg)\wedge\big((x\vee y)'\vee y\big)\wedge(x\vee y)\approx \\
& \approx(x'\vee x)\wedge y\approx1\wedge y\approx y.
\end{align*}
\end{enumerate}
\end{proof}

From (vi), (vii) and (viii) of Lemma~\ref{lem1} it is evident that these terms satisfy (ii) of Proposition~\ref{prop1} for $n=1$. Hence, we can apply now Prposition~\ref{prop2} and we are ready to determine a finite basis of ideals terms of $\mathcal V$.

\begin{theorem}\label{th3}
The ideal terms
\begin{align*}
& t_1\big(t(x_1,x_2,y_1)\vee t(x_3,x_4,y_2),x_2\vee x_4\big), \\
& t_1\big(t(x_1,x_2,y_1)\wedge t(x_3,x_4,y_2),x_2\wedge x_4\big), \\
& t_1\Big(\big(t(x_1,x_2,y)\big)',x_2'\Big), \\
& 1
\end{align*}
form a finite basis of ideal terms of $\mathcal V$.
\end{theorem}

\begin{proof}
This follows from Proposition~\ref{prop2} and Lemma~\ref{lem1}.
\end{proof}

By Corollary~10.3.2 in \cite{CEL}, if $I$ is an ideal of an algebra $\mathbf A=(A,F)$ belonging to a permutable and weakly regular variety then $I$ is the kernel of the congruence
\[
\Theta_I:=\{(x,y)\in A^2\mid t_1(x,y),\ldots,t_n(x,y)\in I\}
\]
on $\mathbf A$ where $t_1,\ldots,t_n$ denote the terms occurring in Proposition~\ref{prop1}.

In our case, we know from Lemma~\ref{lem1} and Theorem~\ref{th3} that for the variety $\mathcal V$, we can take $n=1$ and $t_1(x,y)=(x\wedge y)\vee(x\vee y)'$. Hence, we immediately obtain

\begin{corollary}
If $\mathbf L$ be a member of the variety $\mathcal V$ and $I$ an ideal of $\mathbf L$ then $I$ is the kernel of the congruence
\[
\Theta_I=\{(x,y)\in L^2\mid(x\wedge y)\vee(x\vee y)'\in I\}
\]
on $\mathbf L$.
\end{corollary}

\begin{example}\label{ex2}
Consider the lattice depicted in Fig.~2. For the complementation presented in Example~\ref{ex1}, there are only three ideals, namely $\{1\}$, $I:=\{d,f,g,h,1\}$ and $L:=\{0,a,b,c,d,e,f,g,h,1\}$. We have $\Theta_I=I^2\cup(L\setminus I)^2$. If, however, the complementation is given in another way by
\[
\begin{array}{l|cccccccccc}
x  & 0 & a & b & c & d & e & f & g & h & 1 \\
\hline
x' & 1 & g & h & f & e & d & b & c & a & 0
\end{array}
\]
then this lattice has four ideals, namely $\{1\}$, $\{e,1\}$, $I:=\{d,f,g,h,1\}$ and $L$. We have
\begin{align*}
\Theta_{\{e,1\}} & =\{e,1\}^2\cup\{a,f\}^2\cup\{b,g\}^2\cup\{c,h\}^2\cup\{0,d\}^2, \\
        \Theta_I & =I^2\cup(L\setminus I)^2.
\end{align*}        
\end{example}

\begin{remark}
The variety $\mathcal V$ is residually large, i.e.\ it has infinitely many finite subdirectly irreducible members. Namely, consider the lattice $\mathbf M_n=(M_n,\vee,\wedge)$ for $n\ge3$ visualized in Figure~3

\vspace*{-8mm}

\begin{center}
\setlength{\unitlength}{7mm}
\begin{picture}(6,6)
\put(3,1){\circle*{.3}}
\put(1,3){\circle*{.3}}
\put(5,3){\circle*{.3}}
\put(3,5){\circle*{.3}}
\put(2,3){\circle*{.3}}
\put(3,1){\line(-1,1)2}
\put(3,1){\line(-1,2)1}
\put(3,1){\line(1,1)2}
\put(3,5){\line(-1,-1)2}
\put(3,5){\line(1,-1)2}
\put(3,5){\line(-1,-2)1}
\put(2.85,.3){$0$}
\put(.2,2.85){$a_1$}
\put(2.35,2.85){$a_2\hspace*{3mm}\cdots$}
\put(5.35,2.85){$a_n$}
\put(2.85,5.4){$1$}
\put(2.2,-.75){{\rm Fig.~3}}
\put(.4,-1.75){{\rm Modular lattice $\mathbf M_n$}}
\end{picture}
\end{center}

\vspace*{10mm}

It is well-known that $\mathbf M_n$ is modular and simple and hence subdirectly irreducible. Let $\pi$ be a permutation of the set $\{a_1,\ldots,a_n\}	$ having no fixed points, i.e.\ satisfying $\pi(a_i)\ne a_i$ for all $i=1,\ldots,n$. Define $0':=1$, $1':=0$ and $a_i':=\pi(a_i)$ for all $i=1,\ldots,n$. Then $'$ is a complementation on $\mathbf M_n$ and $(M_n,\vee,\wedge,{}',0,1)$ is a finite subdirectly irreducible member of $\mathcal V$. However, these lattices are not the only finite subdirectly irreducible members of the variety $\mathcal V$. The lattice mentioned in the first case of Example~\ref{ex2} is not isomorphic to $\mathbf M_n$, but it is also a finite subdirectly irreducible member of $\mathcal V$ since it has only one non-trivial congruence which is obviously the only atom in the corresponding congruence lattice.
\end{remark}








Authors' addresses:

V\'aclav Cenker \\
Palack\'y University Olomouc \\
Faculty of Science \\
Department of Algebra and Geometry \\
17.\ listopadu 12 \\
771 46 Olomouc \\
Czech Republic \\
vaclav.cenker01@upol.cz

Ivan Chajda \\
Palack\'y University Olomouc \\
Faculty of Science \\
Department of Algebra and Geometry \\
17.\ listopadu 12 \\
771 46 Olomouc \\
Czech Republic \\
ivan.chajda@upol.cz

Helmut L\"anger \\
TU Wien \\
Faculty of Mathematics and Geoinformation \\
Institute of Discrete Mathematics and Geometry \\
Wiedner Hauptstra\ss e 8-10 \\
1040 Vienna \\
Austria, and \\
Palack\'y University Olomouc \\
Faculty of Science \\
Department of Algebra and Geometry \\
17.\ listopadu 12 \\
771 46 Olomouc \\
Czech Republic \\
helmut.laenger@tuwien.ac.at

\begin{thebibliography}{99}
\bibitem B
L.~Beran, Orthomodular Lattices. Algebraic Approach. Reidel, Dordrecht 1985. ISBN~90-277-1715-X.
\bibitem{BV}
G.~Birkhoff and J.~von~Neumann, The logic of quantum mechanics. Ann.\ of Math.\ {\bf37} (1936), 823--843.
\bibitem{CEL}
I.~Chajda, G.~Eigenthaler and H.~L\"anger, Congruence Classes in Universal Algebra. Heldermann, Lemgo 2012. ISBN~3-88538-226-1.
\bibitem{CL17}
I.~Chajda and H.~L\"anger, Orthomodular lattices can be converted into left residuated l-groupoids. Miskolc Math.\ Notes {\bf18} (2017), 685--689.
\bibitem{CL18}
I.~Chajda and H.~L\"anger, Weakly orthomodular and dually weakly orthomodular lattices. Order {\bf35} (2018), 541--555.
\bibitem{CLa}
I.~Chajda and H.~L\"anger, Algebras and varieties where Sasaki operations form an adjoint pair. Miskolc Math.\ Notes (submitted).
\bibitem C
B.~Cs\'ak\'any, Characterizations of regular varieties. Acta Sci.\ Math.\ (Szeged) {\bf31} (1970), 187--189.
\bibitem H
K.~Husimi, Studies on the foundation of quantum mechanics. I. Proc.\ Phys.-Math.\ Soc.\ Japan {\bf19} (1937), 766--789.
\bibitem N
M.~Nakamura, The permutability in a certain orthocomplemented lattice. Kodai Math.\ Sem.\ Rep.\ {\bf9} (1957), 158--160.
\bibitem S
U.~Sasaki, Orthocomplemented lattices satisfying the exchange axiom. J.\ Sci.\ Hiroshima Univ.\ Ser.\ A {\bf17} (1954), 293--302.
\end{thebibliography}
\end{document}